\documentclass[11pt]{article}
\usepackage[margin=1in]{geometry}
\usepackage{amsmath,amsthm,mathtools}
\usepackage{newtxtext,newtxmath}
\usepackage{graphicx}
\usepackage{booktabs}
\usepackage{microtype}
\usepackage{setspace}
\usepackage[style=authoryear,backend=biber,natbib=true,maxcitenames=2,maxbibnames=99]{biblatex}
\usepackage{hyperref}
\usepackage{enumitem}
\usepackage{caption}
\usepackage{float}
\usepackage{xcolor}

\hypersetup{colorlinks=true,citecolor=black,linkcolor=black,urlcolor=blue}
\newtheorem{theorem}{Theorem}
\newtheorem{proposition}{Proposition}
\newtheorem{corollary}{Corollary}
\theoremstyle{remark}
\newtheorem{remark}{Remark}

\title{\textbf{When Is a Relevance Threshold Statistically Resolvable?}\\[6pt]
Minimax Limits for Effect Classification}
\author{Subir Hait\\
Measurement and Quantitative Methods, College of Education\\
Michigan State University}
\date{}

\begin{document}
\maketitle
\thispagestyle{empty}
\clearpage

\begin{abstract}
Statistical precision and scientific relevance operate on different scales. In regular problems, sampling uncertainty typically contracts at rate $n^{-1/2}$, whereas the magnitude below which an effect is regarded as scientifically negligible may be fixed or may itself vary with information. This paper studies the resulting relative-rate problem.

Let $\Delta_n$ denote a scientific relevance threshold and $I_0$ the Fisher information in a regular scalar model. Define the relevance-resolution index
\[
\lambda_n=\sqrt{nI_0}\,\Delta_n,
\]
which measures the scientific threshold in statistical-error units. For separated negligible and meaningful parameter classes, we establish the minimax lower bound
\[
\liminf_{n\to\infty}R_n^*\ge 2\{1-\Phi(\varepsilon\kappa)\},
\qquad \lambda_n\to\kappa.
\]
If $\lambda_n\to0$, the corresponding experiments merge and no procedure can consistently distinguish scientifically negligible from scientifically meaningful effects. If $\lambda_n\to\kappa\in(0,\infty)$, inference converges to a nondegenerate Gaussian-shift decision problem. We characterize the exact minimax rule and risk in that limiting composite problem. Near the unresolved edge $\kappa\downarrow0$, the optimal cutoff converges to one statistical-error unit and the optimal improvement over trivial risk is of order $\kappa^2$; a simple rule that thresholds at the scientific boundary improves only at order $\kappa^3$. As $\kappa$ grows, that simple rule becomes asymptotically minimax. A complementary achievability result shows that if $\lambda_n\to\infty$, consistent classification is attainable under a uniform estimation-resolution condition, which we verify in Gaussian and Bernoulli models. For $\Delta_n=d n^{-\gamma}$, the critical rate is $\gamma=1/2$.

We further show that moving asymmetric relevance regions are governed by standardized distances to their two boundaries rather than by total width alone, and that differential boundary rates can invalidate conclusions based only on a scalar relevance-resolution index. Finally, for heterogeneous true effects, we derive a significance-saturation result: under a consistent point-null test, the limiting proportion of statistically significant findings that are scientifically negligible is determined by the population mass of nonzero effects inside the relevance region. For the Gaussian point-null rule, the same population mass arises in the opposite low-information limit, revealing a nonmonotone selection phenomenon in which significance can be most selective for scientifically relevant effects at intermediate information.

The framework does not propose a new equivalence test. It characterizes when scientific relevance is statistically resolvable at all, linking effect size, power, local asymptotic theory, minimax decision theory, and practical significance through a common information scale.
\end{abstract}


\clearpage
\section{Introduction}

Increasing information improves statistical precision. Under familiar regularity conditions, estimators stabilize and sampling uncertainty contracts at approximately the $n^{-1/2}$ rate. Confidence intervals narrow, power increases, and increasingly small departures from a point null become statistically detectable.

Scientific importance, however, is not defined by sampling uncertainty. In many applications there is a range of effects sufficiently small to be regarded as clinically, scientifically, economically, educationally, or practically negligible. Let
\[
\mathcal R_n=[-\Delta_n,\Delta_n]
\]
denote such a region, where $\Delta_n>0$ is the smallest magnitude regarded as scientifically meaningful.

The resulting inferential problem contains two scales: statistical resolution and scientific resolution. These scales need not contract at the same rate.

The distinction between statistical significance and substantive importance is old. In an early analysis of the large-sample problem, \citet{berkson1938} warned that departures of little practical consequence can become statistically detectable as information accumulates. \citet{hodges1954} subsequently distinguished statistical from ``material'' significance and developed tests for approximate rather than exact hypotheses. Equivalence and noninferiority methods formalized inference relative to specified margins. More recent methodological guidance has likewise emphasized that statistical significance alone does not quantify substantive importance \citep{wasserstein2016,wasserstein2019}.

The statistics literature has also developed a substantial program of \emph{relevant hypotheses}, replacing exact equality with statements that departures do not exceed a scientifically prespecified threshold. Examples include relevant structural changes in time series and relevant differences in functional data \citep{dette2016,dette2020}.

These developments establish that scientific thresholds can and often should enter the inferential problem explicitly. The present paper addresses a different question:

\begin{quote}
\emph{How small may the scientific relevance threshold become, relative to statistical information, before meaningful and negligible effects cease to be statistically distinguishable?}
\end{quote}

This question is not answered by choosing a particular equivalence test. It is an information question.

Suppose initially that
\[
X_1,\ldots,X_n\stackrel{\mathrm{iid}}{\sim}P_\theta
\]
and the model is regular at $\theta=0$, with Fisher information $0<I_0<\infty$. The central quantity is
\begin{equation}
\lambda_n=\sqrt{nI_0}\,\Delta_n.
\label{eq:lambda}
\end{equation}
The index $\lambda_n$ measures the scientific relevance threshold in local statistical-error units.

The paper makes three main contributions. First, we derive a minimax lower bound for distinguishing scientifically negligible from scientifically meaningful parameter classes:
\begin{equation}
\liminf_{n\to\infty}R_n^*\ge 2\{1-\Phi(\varepsilon\kappa)\},
\qquad \lambda_n\to\kappa.
\label{eq:minimaxbound}
\end{equation}
At $\kappa=0$, no procedure performs asymptotically better than a trivial decision rule. At finite positive $\kappa$, the problem converges to a nondegenerate Gaussian-shift experiment. Inside that limiting experiment we characterize the exact minimax classifier and risk for the separated composite problem. The optimal cutoff approaches one statistical-error unit as $\kappa\downarrow0$, while a simple rule that thresholds at the scientific boundary becomes asymptotically minimax as $\kappa$ grows. A complementary upper result establishes consistent relevance classification when $\lambda_n\to\infty$ and a suitable estimator has uniform statistical resolution. We distinguish throughout between an exact result in the Gaussian limit experiment and finite-$n$ attainability in a general quadratic-mean differentiable model, which requires additional uniform control.

Second, we show that asymmetric or moving scientific regions require a finer geometry. Their behavior is governed by standardized distances from the true parameter to the lower and upper relevance boundaries. A relevance interval can become extremely wide in total statistical-error units while one of its boundaries remains statistically unresolved.

Third, we examine what happens across a population of heterogeneous true effects. As point-null power approaches one for every fixed nonzero parameter, statistical significance eventually saturates the nonzero part of the effect distribution. The fraction of significant results that remain scientifically negligible therefore converges to a quantity determined by the effect distribution itself rather than primarily by the nominal significance level. The path to this limit can be nonmonotone: point-null significance can be most selective for scientifically relevant effects at intermediate information, rather than at either very low or very high information.

The central message is not that increasing sample size is undesirable. More information improves estimation. Rather, greater information exposes the distinction between two inferential questions: $\theta\ne0$ and $|\theta|>\Delta$. The first concerns departure from an exact point. The second concerns departure from a scientific boundary.

\section{Related work and distinction from existing relevance tests}

The use of non-point null hypotheses has a long history. \citet{hodges1954} studied approximate validity of statistical hypotheses and explicitly considered departures too small to have material importance. Early bioequivalence procedures include \citet{anderson1983} and the two one-sided tests procedure of \citet{schuirmann1987}. The connection between intersection--union tests and equivalence confidence sets was clarified by \citet{bergerhsu1996}; \citet{wellek2010} provides a systematic treatment.

Optimality within interval-null and equivalence problems has also been studied extensively. \citet{brown1995} develop confidence-set geometry for bioequivalence, \citet{brown1997} construct an unbiased test that can improve on standard two one-sided tests, and \citet{romano2005} develops asymptotic optimality theory by approximating regular statistical experiments with limiting normal experiments. Romano's objective is to construct optimal tests when the equivalence interval is part of the inferential specification.

A particularly close decision-theoretic precedent is \citet{blanchard2018}, who study Gaussian mean testing of a closed convex null set against alternatives separated from that set by a prescribed Euclidean distance. Their criterion is the same worst-case sum of Type I and Type II errors used below, and their principal target is the minimax separation rate as a function of dimension, noise level, and the geometry of the null set. At fixed $\kappa$, our Gaussian critical problem is the one-dimensional interval-null specialization of their setup with $\mathcal C=[-a,a]$ and separation distance $\rho=b-a$. Our emphasis is different: we track how the scientific margin itself moves relative to statistical information, characterize the three relative-rate regimes, and use the exact scalar Gaussian solution to sharpen the finite-$\kappa$ behavior and compare a scientifically natural boundary rule with the minimax rule. The Gaussian optimization is therefore classical in spirit; the contribution is its role inside the relevance-resolution framework and the resulting comparison of scientific and statistical scales.

The present question is therefore not whether interval hypotheses can be tested optimally in a fixed formulation. We allow the relevance region itself to vary with information and ask whether that region remains statistically resolvable.

A related line of statistics research replaces exact equality by \emph{relevant hypotheses}. \citet{dette2016}, for example, consider change-point hypotheses in which a structural change must exceed a specified threshold to count as relevant. In regression settings, \citet{dette2018} formulate equivalence through a prespecified distance threshold, with later extensions to regression curves sharing common parameters \citep{mollenhoff2020}. \citet{dette2020} develop relevant-hypothesis methods for functional time series. These methods demonstrate the practical and theoretical value of inference relative to nonzero scientific margins.

\citet{berger1987} examined difficulties associated with precise hypotheses, while \citet{perlman1999} emphasized that the quality of a testing procedure depends on the inferential problem and decision criterion rather than on simplistic comparisons of rejection probabilities.

Applied methodological work has made relevance-based inference increasingly accessible. \citet{murphy1999} developed minimum-effect tests in the general linear model. \citet{lakens2018} provided practical guidance for equivalence testing using a smallest effect size of interest. \citet{blume2018} proposed second-generation $p$-values based on the relation between a data-supported interval and an interval null. \citet{lakens2022} explicitly relates sample-size justification to effects researchers consider informative, and \citet{riesthuis2024} develops power analysis for minimum-effect and equivalence questions. \citet{ongaro2024} provide a general framework for testing practical relevance through interval null hypotheses.

The contribution here is therefore not the observation that scientific importance differs from statistical significance, nor the proposal of another interval-null test. Instead, we isolate the relative-rate problem
\[
\frac{\text{scientific threshold}}{\text{statistical uncertainty}}
\]
and study the point at which the corresponding scientific classification problem becomes information-theoretically impossible, locally informative, or consistently resolvable.

\section{Statistical and scientific resolution}

\subsection{Scientific relevance regions}

Let
\[
\mathcal R_n=[-\Delta_n,\Delta_n],\qquad \Delta_n>0,
\]
denote the region of scientifically negligible effects. An effect is scientifically negligible when $|\theta|\le\Delta_n$ and scientifically meaningful when $|\theta|>\Delta_n$.

For descriptive purposes, suppose an estimator $\hat\theta_n$ has statistical-error scale $s_n$. Define
\[
\lambda_n=\frac{\Delta_n}{s_n},
\qquad
u_n=\frac{\theta_n}{s_n}.
\]
If $\theta_n=r_n\Delta_n$, then $u_n=r_n\lambda_n$ with $r_n=\theta_n/\Delta_n$. Thus point-null signal strength decomposes into a scientific location coordinate and a relevance-resolution coordinate.

For regular efficient estimation,
\[
s_n\sim \frac{1}{\sqrt{nI_0}},
\]
so the descriptive ratio $\Delta_n/s_n$ becomes the Fisher-standardized index in \eqref{eq:lambda}.

\subsection{Confidence-set classification}

Consider a symmetric interval
\[
C_n=[\hat\theta_n-cs_n,\hat\theta_n+cs_n],
\]
where $c>0$. We say that the data \emph{establish meaningfulness} when $C_n\cap\mathcal R_n=\emptyset$ and \emph{establish negligibility} when $C_n\subset\mathcal R_n$. Otherwise the result is unresolved.

Under the standardized Gaussian representation $Y\sim N(u,1)$, the probability of establishing meaningfulness is
\begin{equation}
P_M(u,\lambda)
=
\Phi[-(\lambda+c)-u]
+
\Phi[u-(\lambda+c)],
\label{eq:pm}
\end{equation}
whereas the probability of establishing negligibility is
\begin{equation}
P_N(u,\lambda)
=
\left[
\Phi\{(\lambda-c)-u\}
-
\Phi\{-(\lambda-c)-u\}
\right]_+.
\label{eq:pn}
\end{equation}
The unresolved probability is $P_U(u,\lambda)=1-P_M(u,\lambda)-P_N(u,\lambda)$. These are operating characteristics of this particular classifier; they are not minimax information bounds.

\subsection{The effect-information plane}

Using $u=r\lambda$, equations \eqref{eq:pm}--\eqref{eq:pn} can be represented in the $(r,\lambda)$ plane. The scientific relevance boundaries are $r=-1$ and $r=1$.

If a conventional two-sided point-null test rejects when $|\hat\theta_n|>cs_n$, then its approximate significance frontier on the same scale is
\[
|r|=\frac{c}{\lambda}.
\]
As $\lambda$ increases, this frontier collapses toward zero while the scientific boundaries remain fixed at $\pm1$. Hence the region $c/\lambda<|r|<1$ contains effects that are statistically detectable by the point-null procedure while remaining scientifically negligible.

The confidence-set classifier also has a negligibility floor: $\lambda>c$. If $\lambda\le c$, an interval with half-width $cs_n$ cannot fit entirely inside a relevance interval with half-width $\Delta_n$.

\begin{figure}[H]
\centering
\includegraphics[width=0.96\linewidth]{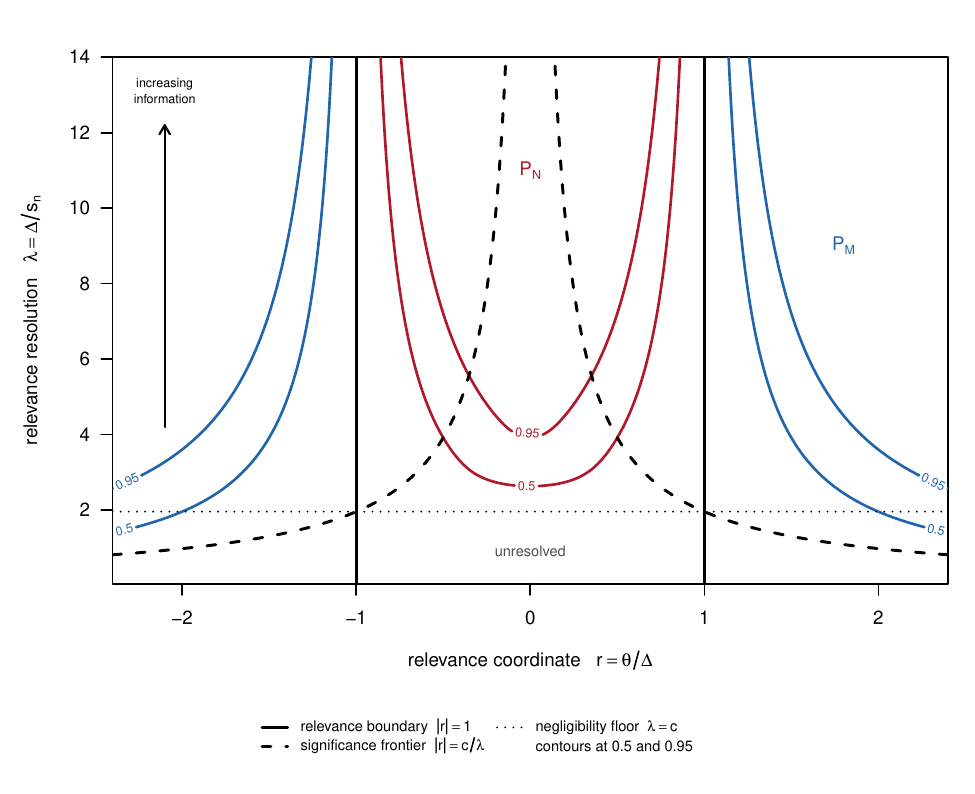}
\caption{\textbf{Effect-information plane.} Probability contours for the confidence-interval classifier in the $(r,\lambda)$ plane, where $r=\theta/\Delta$ and $\lambda=\Delta/s$. Solid vertical lines mark $|r|=1$; dashed curves give the point-null significance frontier $|r|=c/\lambda$; and the dotted horizontal line marks the classifier-specific negligibility floor $\lambda=c$. The contours are finite-information properties of the confidence-interval classifier, not minimax bounds. The asymptotic vertical regimes receive an information-theoretic interpretation in Section~\ref{sec:boundary}.}
\label{fig:plane}
\end{figure}

\section{A statistical resolution boundary}
\label{sec:boundary}

The finite-procedure geometry above motivates a more fundamental question: when can \emph{any} statistical procedure distinguish scientifically negligible effects from scientifically meaningful effects?

Fix $\varepsilon\in(0,1)$. Define separated parameter classes
\[
\mathcal N_n(\varepsilon)
=
\{\theta:|\theta|\le(1-\varepsilon)\Delta_n\},
\qquad
\mathcal M_n(\varepsilon)
=
\{\theta:|\theta|\ge(1+\varepsilon)\Delta_n\}.
\]
The separation by $\varepsilon\Delta_n$ removes the unavoidable ambiguity exactly at the scientific boundary.

Let $\phi_n\in[0,1]$ denote the probability of classifying an observation as scientifically meaningful. Define its worst-case sum-of-errors risk by
\[
R_n(\phi_n)
=
\sup_{\theta\in\mathcal N_n(\varepsilon)}E_\theta\phi_n
+
\sup_{\theta\in\mathcal M_n(\varepsilon)}E_\theta(1-\phi_n),
\]
and let $R_n^*=\inf_{\phi_n}R_n(\phi_n)$. A value $R_n^*=1$ means that asymptotically no procedure improves on a trivial constant decision rule. Consistent classification requires $R_n^*\to0$.

\begin{theorem}[Minimax relevance-resolution lower bound]
\label{thm:minimax}
Suppose $X_1,\ldots,X_n\stackrel{\mathrm{iid}}{\sim}P_\theta$, and the model is differentiable in quadratic mean at $\theta=0$, with Fisher information $0<I_0<\infty$. Let
\[
\lambda_n=\sqrt{nI_0}\,\Delta_n\longrightarrow\kappa\in[0,\infty).
\]
Then, for every fixed $\varepsilon\in(0,1)$,
\[
\liminf_{n\to\infty}R_n^*
\ge
2\{1-\Phi(\varepsilon\kappa)\}.
\]
\end{theorem}

\begin{proof}
Consider the two parameter sequences
\[
\theta_n^-=(1-\varepsilon)\Delta_n\in\mathcal N_n(\varepsilon),
\qquad
\theta_n^+=(1+\varepsilon)\Delta_n\in\mathcal M_n(\varepsilon).
\]
For every decision rule $\phi_n$,
\[
R_n(\phi_n)
\ge
E_{\theta_n^-}\phi_n+E_{\theta_n^+}(1-\phi_n).
\]
Let $h_n^\pm=\sqrt n\,\theta_n^\pm$. Because $\sqrt{nI_0}\Delta_n\to\kappa$,
\[
h_n^\pm\longrightarrow h^\pm
=
\frac{(1\pm\varepsilon)\kappa}{\sqrt{I_0}}.
\]
Quadratic-mean differentiability implies convergence of the corresponding local experiments to the Gaussian-shift experiment
\[
Y\sim N(h,I_0^{-1}).
\]
Standardizing by Fisher information,
\[
T=\sqrt{I_0}\,Y\sim N(t,1),
\qquad t=\sqrt{I_0}\,h,
\]
so the two limiting means are $t^-=(1-\varepsilon)\kappa$ and $t^+=(1+\varepsilon)\kappa$, with separation $2\varepsilon\kappa$.

The reduction to two simple hypotheses is the standard two-point device used in minimax lower bounds \citep[Chapter~2]{tsybakov2009}. Convergence of the two-point local experiments in Le Cam's sense implies convergence of the optimal bounded-loss risk for the corresponding simple testing problem; equivalently, the minimal sum of Type I and Type II errors converges to that of the two-point Gaussian limit experiment \citep[Chapter~9]{vandervaart1998}; see also \citet{lecam2000}. For two unit-variance normal distributions with mean separation $2\varepsilon\kappa$,
\[
\|N(t^-,1)-N(t^+,1)\|_{\mathrm{TV}}
=
2\Phi(\varepsilon\kappa)-1.
\]
Therefore the limiting minimal sum of errors for the two-point problem is $2\{1-\Phi(\varepsilon\kappa)\}$. Because the two-point problem is contained within the composite problem, the result follows.
\end{proof}

\begin{remark}[The role of the separation parameter]
The right-hand side of Theorem~\ref{thm:minimax} decreases with $\varepsilon$. This does not make $\varepsilon$ a tuning parameter to be optimized after the fact. Different values of $\varepsilon$ define different scientific classification problems by changing the distance between the negligible and meaningful parameter classes. The theorem should therefore be interpreted for a prespecified fixed separation margin.
\end{remark}

\begin{remark}[Oscillating relevance thresholds]
The convergence assumption in Theorem~\ref{thm:minimax} is used to identify a single limit experiment. If $\lambda_n$ does not converge, the theorem applies subsequence-wise. In particular, if $\underline\kappa=\liminf_n\lambda_n<\infty$, there exists a subsequence along which $\lambda_n\to\underline\kappa$, so
\[
\limsup_{n\to\infty}R_n^*\ge 2\{1-\Phi(\varepsilon\underline\kappa)\}.
\]
If $\overline\kappa=\limsup_n\lambda_n<\infty$, the same subsequence argument yields the conservative global bound
\[
\liminf_{n\to\infty}R_n^*\ge 2\{1-\Phi(\varepsilon\overline\kappa)\}.
\]
Thus oscillation can be handled without assigning a single artificial resolution regime.
\end{remark}

\subsection{The subcritical regime: impossibility}

If $\lambda_n\to0$, Theorem~\ref{thm:minimax} gives $\liminf_nR_n^*\ge1$. Because the constant rule $\phi_n\equiv0$ has risk exactly one,
\[
R_n^*\to1.
\]
Thus the two scientific classes are asymptotically indistinguishable.

For this regime there is also a direct proof that avoids the Gaussian limit experiment. Define squared Hellinger distance by
\[
h^2(P,Q)=1-\int\sqrt{dP\,dQ}.
\]
Quadratic-mean differentiability gives
\[
h^2(P_\theta,P_0)=\frac18I_0\theta^2\{1+o(1)\}.
\]
For product measures,
\[
h^2(P_\theta^n,P_0^n)
=
1-\{1-h^2(P_\theta,P_0)\}^n.
\]
Take $\theta_n=(1+\varepsilon)\Delta_n$. If $\sqrt n\Delta_n\to0$, then $n h^2(P_{\theta_n},P_0)\to0$ and hence $h(P_{\theta_n}^n,P_0^n)\to0$. Using the convention
\[
\|P-Q\|_{\mathrm{TV}}=\sup_A|P(A)-Q(A)|,
\]
we have
\[
\|P-Q\|_{\mathrm{TV}}\le\sqrt2\,h(P,Q),
\]
and therefore
\[
\|P_{\theta_n}^n-P_0^n\|_{\mathrm{TV}}\to0.
\]
The experiments literally merge.

This result is stronger than an inability to distinguish effects lying just beyond the relevance boundary. Let
\[
b_n=(\sqrt n\Delta_n)^{-1/2}.
\]
Whenever $\sqrt n\Delta_n\to0$, we have $b_n\to\infty$ but
\[
\sqrt n\,b_n\Delta_n=(\sqrt n\Delta_n)^{1/2}\to0.
\]
Thus $\theta_n=b_n\Delta_n$ satisfies $|\theta_n|/\Delta_n\to\infty$ while remaining asymptotically indistinguishable from zero. An effect may therefore exceed the scientific threshold by a diverging factor and still be statistically unresolved if the scientific threshold itself contracts sufficiently faster than the experiment's information scale.

For the confidence-interval classifier of Section~3, this information-theoretic impossibility has a simple procedural shadow. If $r$ is fixed and $\lambda_n\to0$, then $u_n=r\lambda_n\to0$, so
\[
P_M\to2\Phi(-c).
\]
For $c=z_{1-\alpha/2}$,
\begin{equation}
P_M\to\alpha.
\label{eq:alpha-shadow}
\end{equation}
The classifier declares meaningfulness only at its nominal false-positive rate, exactly as one would expect in a regime with no asymptotic discriminatory information.

\subsection{The critical regime: finite scientific information}

Suppose $\lambda_n\to\kappa\in(0,\infty)$. Theorem~\ref{thm:minimax} implies
\[
\liminf_nR_n^*\ge2\{1-\Phi(\varepsilon\kappa)\}>0.
\]
Therefore no procedure can classify the separated scientific classes consistently. The limiting experiment nevertheless contains genuine statistical information. We next solve the corresponding one-dimensional Gaussian composite problem exactly.

For $\kappa>0$ and $0<\varepsilon<1$, define
\[
a=(1-\varepsilon)\kappa,
\qquad
b=(1+\varepsilon)\kappa,
\qquad 0<a<b.
\]
In the experiment $Y\sim N(\mu,1)$, consider
\[
\mathcal N_\infty=\{\mu:|\mu|\le a\},
\qquad
\mathcal M_\infty=\{\mu:|\mu|\ge b\},
\]
and write
\[
R_\infty(\phi)
=
\sup_{|\mu|\le a}E_\mu\phi(Y)
+
\sup_{|\mu|\ge b}E_\mu\{1-\phi(Y)\},
\qquad
R_\infty^*=\inf_\phi R_\infty(\phi).
\]

\begin{proposition}[Exact minimax characterization in the critical Gaussian experiment]
\label{prop:critical-exact}
Let $g_m$ denote the equal mixture of $N(m,1)$ and $N(-m,1)$,
\[
g_m(y)=\frac12\{\varphi(y-m)+\varphi(y+m)\}
=\varphi(y)e^{-m^2/2}\cosh(my).
\]
There is a unique $c^*=c^*(\kappa,\varepsilon)>0$ satisfying
\begin{equation}
e^{-b^2/2}\cosh(bc^*)
=
e^{-a^2/2}\cosh(ac^*).
\label{eq:cstar}
\end{equation}
The rule
\begin{equation}
\phi^*(Y)=\mathbf1\{|Y|>c^*\}
\label{eq:phistar}
\end{equation}
is minimax for $\mathcal N_\infty$ versus $\mathcal M_\infty$. Its exact risk is
\begin{equation}
\begin{aligned}
R_\infty^*(\kappa,\varepsilon)
={}&\Phi(a-c^*)+\Phi(-a-c^*)\\
&+\Phi(c^*-b)-\Phi(-c^*-b).
\end{aligned}
\label{eq:exactrisk}
\end{equation}
\end{proposition}

\begin{proof}
For any rule $\phi$, the worst-case risk dominates the average risk under equal boundary mixtures:
\[
R_\infty(\phi)
\ge
E_{g_a}\phi+E_{g_b}(1-\phi).
\]
The infimum of the right-hand side is the simple-versus-simple Bayes risk for $g_a$ against $g_b$, attained by the likelihood-ratio rule that chooses the meaningful class when $g_b(y)>g_a(y)$.

For $x=|y|$,
\[
\frac{g_b(y)}{g_a(y)}
=
e^{-(b^2-a^2)/2}
\frac{\cosh(bx)}{\cosh(ax)}.
\]
For $x>0$, the derivative of its log is
\[
b\tanh(bx)-a\tanh(ax)>0,
\]
because $m\mapsto m\tanh(mx)$ is strictly increasing for $m>0$. At $x=0$ the ratio is smaller than one, whereas it diverges as $x\to\infty$. Hence there is a unique crossing $c^*>0$, characterized by \eqref{eq:cstar}, and the Bayes rule is \eqref{eq:phistar}.

It remains to verify that the boundary mixtures are least favorable for the composite problem. For any fixed $c>0$ and $\mu\ge0$,
\[
P_\mu(|Y|>c)=\Phi(\mu-c)+\Phi(-\mu-c)
\]
is increasing in $\mu$. Therefore the largest false-positive probability over $|\mu|\le a$ occurs at $|\mu|=a$, and the largest false-negative probability over $|\mu|\ge b$ occurs at $|\mu|=b$. For the likelihood-ratio threshold $c^*$, the composite worst-case risk consequently equals the boundary-mixture Bayes risk. The lower bound above is therefore attained, proving minimaxity. Substituting the two boundary error probabilities yields \eqref{eq:exactrisk}.
\end{proof}

\begin{corollary}[The unresolved-edge expansion]
\label{cor:smallkappa}
As $\kappa\downarrow0$,
\begin{equation}
c^*
=
1+\frac{1+\varepsilon^2}{6}\kappa^2+O(\kappa^4),
\label{eq:cstar-small}
\end{equation}
and
\begin{equation}
R_\infty^*(\kappa,\varepsilon)
=
1-4\varepsilon\varphi(1)\kappa^2+O(\kappa^4).
\label{eq:rstar-small}
\end{equation}
\end{corollary}

The limiting cutoff in \eqref{eq:cstar-small} has a useful interpretation. As the relevance scale collapses toward zero, the optimal rule does \emph{not} chase the scientific boundary to zero; it stabilizes at one statistical-error unit. Indeed, for small $m$,
\[
\frac{g_m(y)}{\varphi(y)}
=
1+\frac{m^2}{2}(y^2-1)+O(m^4),
\]
so the leading local contrast between the two symmetric boundary mixtures changes sign at $|y|=1$. This second-order structure explains both the one-unit cutoff and the quadratic improvement in \eqref{eq:rstar-small}.

The exact minimax rule is informative, but the simpler scientific-boundary threshold remains useful.

\begin{proposition}[The scientific-boundary rule]
\label{prop:critical-simple}
Define
\[
\phi_\kappa(Y)=\mathbf1\{|Y|>\kappa\}.
\]
Its worst-case sum-of-errors risk is
\begin{equation}
S_\varepsilon(\kappa)
=
2\Phi(-\varepsilon\kappa)
+
\Phi\{-(2-\varepsilon)\kappa\}
-
\Phi\{-(2+\varepsilon)\kappa\}.
\label{eq:S}
\end{equation}
Moreover, $S_\varepsilon(\kappa)<1$ for every $\kappa>0$, and, as $\kappa\downarrow0$,
\begin{equation}
S_\varepsilon(\kappa)
=
1-4\varepsilon\varphi(0)\kappa^3+O(\kappa^5).
\label{eq:cubic}
\end{equation}
\end{proposition}

\begin{proof}
For $\mu\ge0$, the rejection probability $P_\mu(|Y|>\kappa)$ is increasing in $\mu$. Hence the worst false-positive probability over the negligible class occurs at $\mu=a=(1-\varepsilon)\kappa$ and equals
\[
\Phi(-\varepsilon\kappa)+\Phi\{-(2-\varepsilon)\kappa\}.
\]
Likewise, the worst false-negative probability over the meaningful class occurs at $\mu=b=(1+\varepsilon)\kappa$ and equals
\[
\Phi(-\varepsilon\kappa)-\Phi\{-(2+\varepsilon)\kappa\}.
\]
Adding gives \eqref{eq:S}.

Furthermore,
\[
\begin{aligned}
1-S_\varepsilon(\kappa)
&=
\int_{-\varepsilon\kappa}^{\varepsilon\kappa}\varphi(x)\,dx
-
\int_{(2-\varepsilon)\kappa}^{(2+\varepsilon)\kappa}\varphi(t)\,dt\\
&=
\int_{-\varepsilon\kappa}^{\varepsilon\kappa}
\{\varphi(x)-\varphi(x+2\kappa)\}\,dx>0,
\end{aligned}
\]
because $|x+2\kappa|>|x|$ on the integration interval. Taylor expansion around zero gives \eqref{eq:cubic}.
\end{proof}

Corollary~\ref{cor:smallkappa} and Proposition~\ref{prop:critical-simple} show that the simple rule loses one order of improvement only near the unresolved edge:
\[
1-R_\infty^*(\kappa,\varepsilon)=O(\kappa^2),
\qquad
1-S_\varepsilon(\kappa)=O(\kappa^3).
\]
This is a local statement as $\kappa\downarrow0$, not a general indictment of the simple rule. It also shows why the two-point lower bound in Theorem~\ref{thm:minimax}, although sufficient for the phase transition, is not locally sharp in the composite critical problem: as $\kappa\downarrow0$,
\[
2\{1-\Phi(\varepsilon\kappa)\}
=
1-2\varepsilon\varphi(0)\kappa+O(\kappa^3),
\]
whereas the exact composite minimax risk departs from one only at order $\kappa^2$. The boundary mixtures, rather than a single same-sign pair of points, capture the correct unresolved-edge difficulty.

\begin{corollary}[Large-$\kappa$ recovery]
\label{cor:largekappa}
As $\kappa\to\infty$ with $\varepsilon\in(0,1)$ fixed,
\[
\frac{c^*}{\kappa}\to1,
\qquad
R_\infty^*(\kappa,\varepsilon)
\sim
2\Phi(-\varepsilon\kappa),
\qquad
S_\varepsilon(\kappa)
\sim
R_\infty^*(\kappa,\varepsilon).
\]
Hence the scientific-boundary rule is asymptotically minimax as the critical Gaussian problem itself becomes increasingly well resolved.
\end{corollary}

The threshold equation gives a sharper description. Using
$\log\cosh z=z-\log2+\log(1+e^{-2z})$ for $z>0$, equation \eqref{eq:cstar} implies $c^*>\kappa$ and
\[
0<c^*-\kappa
\le
\frac{e^{-2(1-\varepsilon)\kappa^2}}{2\varepsilon\kappa}.
\]
Thus the optimal and scientific-boundary cutoffs become exponentially close on the large-$\kappa$ scale.

\begin{remark}[Limit experiment versus finite-$n$ minimax risk]
\label{rem:transfer}
Proposition~\ref{prop:critical-exact} is an exact characterization of the \emph{limiting Gaussian experiment}. Theorem~\ref{thm:minimax} transfers a lower bound from that local experiment under quadratic-mean differentiability. A matching general statement $R_n^*\to R_\infty^*$ does not follow from pointwise LAN alone. The meaningful class $\mathcal M_n(\varepsilon)$ extends beyond bounded local neighborhoods, so an upper transfer additionally requires uniform control in the local region together with a condition ensuring that more distant alternatives are no harder. We therefore do not claim general finite-$n$ attainability under quadratic-mean differentiability alone.
\end{remark}

\begin{corollary}[Exact Gaussian-location realization]
\label{cor:gaussian-finite}
Suppose $X_i\stackrel{\mathrm{iid}}{\sim}N(\theta,\sigma^2)$ with known $\sigma$. Let
\[
\lambda_n=\frac{\sqrt n\,\Delta_n}{\sigma},
\qquad
Y_n=\frac{\sqrt n\,\bar X}{\sigma}.
\]
For every $n$, the separated classes map exactly to
\[
|\mu|\le(1-\varepsilon)\lambda_n
\qquad\text{and}\qquad
|\mu|\ge(1+\varepsilon)\lambda_n,
\]
where $Y_n\sim N(\mu,1)$. Consequently the threshold rule
\[
\phi_n^*=\mathbf1\{|Y_n|>c^*(\lambda_n,\varepsilon)\}
\]
is finite-sample minimax, with risk equal to \eqref{eq:exactrisk} after replacing $\kappa$ by $\lambda_n$. In particular, if $\lambda_n\to\kappa\in(0,\infty)$, then $R_n^*\to R_\infty^*(\kappa,\varepsilon)$ exactly in this model.
\end{corollary}

\subsection{Operating characteristics at the boundary rate}

Suppose $\theta_n=r\Delta_n$ with fixed $r$ and $\lambda_n\to\kappa$. Then $u_n\to r\kappa$. Consequently, the confidence-set classifier has limiting probabilities $P_M(r\kappa,\kappa)$, $P_N(r\kappa,\kappa)$, and $P_U(r\kappa,\kappa)$ given by Section~3.2. The boundary-rate experiment is therefore the formal asymptotic interpretation of a horizontal slice through the finite-$\lambda$ effect-information plane. Neither statistical uncertainty nor the scientific threshold dominates; they remain comparable.

\subsection{The supercritical regime: achievability}

Theorem~\ref{thm:minimax} is a finite-local-information result and is stated for $\lambda_n\to\kappa<\infty$. Its lower-bound function $2\{1-\Phi(\varepsilon\kappa)\}$ vanishes as the finite local information level $\kappa$ grows, so the corresponding two-point argument becomes asymptotically uninformative in a supercritical sequence $\lambda_n\to\infty$. A separate upper result shows that consistent resolution is then achievable under an appropriate uniform estimation condition.

\paragraph{Uniform Resolution Condition.}
Suppose an estimator $\hat\theta_n$ and deterministic scale $s_n\downarrow0$ satisfy
\begin{equation}
\lim_{M\to\infty}\limsup_{n\to\infty}
\sup_{\theta\in\mathcal N_n(\varepsilon)\cup\mathcal M_n(\varepsilon)}
P_\theta\bigl(|\hat\theta_n-\theta|>Ms_n\bigr)=0.
\label{eq:urc}
\end{equation}
If a data-dependent scale $\hat s_n$ is used, assume also that $\hat s_n/s_n\to1$ uniformly in probability over the same parameter classes; when the scale is deterministic, set $\hat s_n=s_n$.

\begin{proposition}[Uniform achievability]
\label{prop:uniform}
Under the Uniform Resolution Condition, if
\[
\frac{\Delta_n}{s_n}\to\infty,
\]
then, for every fixed $0\le c<\infty$, the interval
\[
C_n=[\hat\theta_n-c\hat s_n,\hat\theta_n+c\hat s_n]
\]
satisfies
\[
\inf_{\theta\in\mathcal N_n(\varepsilon)}P_\theta(C_n\subset\mathcal R_n)\to1
\]
and
\[
\inf_{\theta\in\mathcal M_n(\varepsilon)}P_\theta(C_n\cap\mathcal R_n=\emptyset)\to1.
\]
\end{proposition}

\begin{proof}
Every point in $\mathcal N_n(\varepsilon)$ lies at least $\varepsilon\Delta_n$ from the closest boundary of $\mathcal R_n$. The inclusion $C_n\subset\mathcal R_n$ therefore holds whenever
\[
|\hat\theta_n-\theta|+c\hat s_n<\varepsilon\Delta_n.
\]
Uniform tightness at scale $s_n$, together with $s_n=o(\Delta_n)$, makes the probability of this event converge uniformly to one. The same argument applies to $\mathcal M_n(\varepsilon)$, because every meaningful parameter lies at least $\varepsilon\Delta_n$ beyond its closest scientific boundary.
\end{proof}

\paragraph{Two model checks.}
The Uniform Resolution Condition is deliberately stated separately from pointwise asymptotic normality because it must be verified for the model and parameter range at hand. Two elementary examples show that the condition is not vacuous.

First, if $X_i\stackrel{\mathrm{iid}}{\sim}N(\theta,\sigma^2)$ with known $\sigma$, take $\hat\theta_n=\bar X$ and $s_n=\sigma/\sqrt n$. Then
\[
\frac{\hat\theta_n-\theta}{s_n}\sim N(0,1)
\]
for every $n$ and every $\theta$. Hence
\[
\sup_{\theta}P_\theta\bigl(|\hat\theta_n-\theta|>Ms_n\bigr)
=2\Phi(-M)\longrightarrow0
\]
as $M\to\infty$, so \eqref{eq:urc} holds globally and exactly.

Second, let $X_i\stackrel{\mathrm{iid}}{\sim}\mathrm{Bernoulli}(p)$ and parameterize $\theta=p-p_0$ for a fixed reference value $p_0\in(0,1)$. With $\hat\theta_n=\bar X-p_0$ and the deterministic scale $s_n=n^{-1/2}$, Hoeffding's inequality gives
\[
\sup_{0\le p\le1}
P_p\bigl(|\hat\theta_n-\theta|>M n^{-1/2}\bigr)
\le 2e^{-2M^2}.
\]
Thus \eqref{eq:urc} also holds uniformly in this non-Gaussian model. The deterministic scale $n^{-1/2}$ differs only by a fixed positive factor from the Fisher-standardized scale used in the relevance-resolution index: at $p=p_0$, $I_0=\{p_0(1-p_0)\}^{-1}$ and hence $I_0^{-1/2}n^{-1/2}=\sqrt{p_0(1-p_0)}\,n^{-1/2}$. Rescaling $M$ by this constant therefore gives the same Uniform Resolution Condition on the Fisher-information scale. These examples are illustrations rather than a generic theorem: for other regular models, especially with estimated scales or unbounded parameter spaces, the required uniformity must be checked rather than inferred from pointwise asymptotic normality.

\begin{corollary}[The square-root boundary]
\label{cor:sqrt}
For a regular estimator satisfying $s_n\asymp n^{-1/2}$, let
\[
\Delta_n=d n^{-\gamma},\qquad d>0.
\]
Then $\lambda_n\asymp n^{1/2-\gamma}$. Therefore $\gamma>1/2$ is the information-theoretic impossibility regime, $\gamma=1/2$ is the nondegenerate local regime, and $\gamma<1/2$ is consistently resolvable under the Uniform Resolution Condition. Equivalently, $\Delta_n=o(n^{-1/2})$ cannot be resolved consistently, $\Delta_n\asymp n^{-1/2}$ produces finite limiting scientific information, and $\Delta_n\gg n^{-1/2}$ is consistently resolvable under a uniformly rate-consistent estimator.
\end{corollary}

\begin{figure}[H]
\centering
\includegraphics[width=0.98\linewidth]{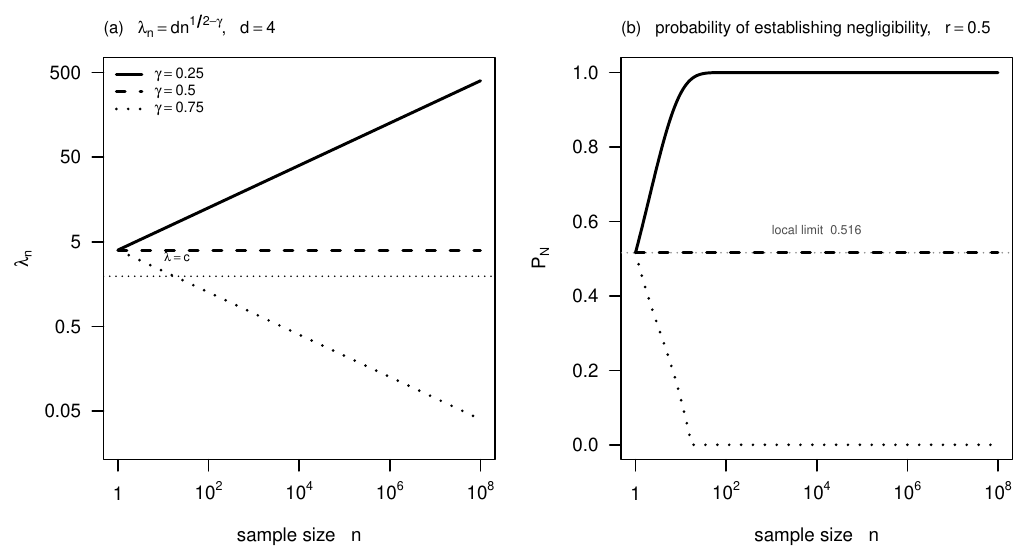}
\caption{\textbf{Relevance-resolution phase transition.} For thresholds $\Delta_n=d n^{-\gamma}$, the regular relevance-resolution index satisfies $\lambda_n\propto n^{1/2-\gamma}$. Panel (a) displays trajectories below, at, and above the critical rate $\gamma=1/2$. Panel (b) shows the corresponding confidence-set classification probabilities. At $\gamma=1/2$, the limiting probability is nondegenerate because the problem converges to a finite Gaussian-shift experiment.}
\label{fig:phase}
\end{figure}

\section{Statistical significance and scientific negligibility can both become certain}

The resolution framework clarifies a familiar large-sample phenomenon without treating it as a paradox. Suppose the relevance threshold is fixed, $\Delta_n=\Delta>0$, and let
\[
0<|\theta|<\Delta.
\]
The true effect is nonzero but scientifically negligible.

For any consistent point-null test of $H_0:\theta=0$,
\[
P_\theta(\text{reject }H_0)\to1.
\]
To apply Proposition~\ref{prop:uniform} explicitly, define
\[
\varepsilon_\theta=1-\frac{|\theta|}{\Delta}.
\]
Because $0<|\theta|<\Delta$, we have $\varepsilon_\theta\in(0,1)$ and $|\theta|=(1-\varepsilon_\theta)\Delta$, so $\theta\in\mathcal N_n(\varepsilon_\theta)$ for every $n$. Since a fixed $\Delta$ satisfies $\Delta/s_n\to\infty$ under regular increasing information, Proposition~\ref{prop:uniform} gives
\[
P_\theta\bigl(C_n\subset[-\Delta,\Delta]\bigr)\to1.
\]
Hence
\begin{equation}
P_\theta(\text{point-null significant and scientifically negligible})\to1.
\label{eq:simultaneous}
\end{equation}
The two conclusions are not contradictory. The point-null test asks whether $\theta=0$; the relevance classifier asks whether $|\theta|\le\Delta$. Both questions can eventually be answered with near certainty.

This observation also explains why merely changing the significance level does not solve the scientific problem. Replacing $\alpha$ by $\alpha_n$ changes the evidential requirement for distinguishing a parameter from zero. Introducing $\Delta$ changes the scientific hypothesis itself.

\begin{figure}[H]
\centering
\includegraphics[width=0.96\linewidth]{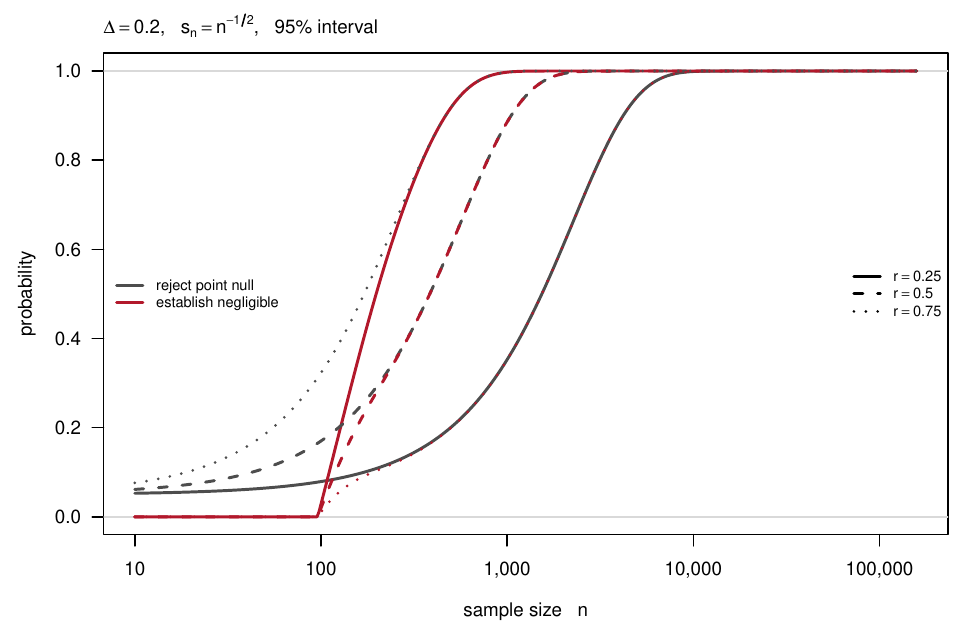}
\caption{\textbf{Simultaneous certainty.} For effects strictly inside the scientific relevance region but different from zero, both the probability of rejecting a point null and the probability of establishing negligibility converge to one. The two conclusions answer distinct questions. The illustration uses $\Delta=0.20$, $s_n=n^{-1/2}$, and a 95\% interval.}
\label{fig:simultaneous}
\end{figure}

\section{Moving and asymmetric relevance regions}

The scalar index $\lambda_n$ is complete for a symmetric relevance interval centered at the point around which the local experiment is developed. It is not sufficient for general moving asymmetric regions.

Let
\[
\mathcal R_n=[-\Delta_{L,n},\Delta_{U,n}].
\]
Define its half-width and center by
\[
h_n=\frac{\Delta_{L,n}+\Delta_{U,n}}{2},
\qquad
m_n=\frac{\Delta_{U,n}-\Delta_{L,n}}{2}.
\]
For sampling-error scale $s_n$, define
\[
\lambda_n=\frac{h_n}{s_n},
\qquad
u_n=\frac{\theta_n-m_n}{s_n}.
\]
Equivalently, define the standardized distances to the two relevance boundaries:
\begin{equation}
 d_{L,n}=\frac{\theta_n+\Delta_{L,n}}{s_n}=\lambda_n+u_n,
 \qquad
 d_{U,n}=\frac{\Delta_{U,n}-\theta_n}{s_n}=\lambda_n-u_n.
\label{eq:boundarydist}
\end{equation}
These two quantities are the fundamental coordinates.

Under a Gaussian approximation,
\begin{equation}
P_M=\Phi(-c-d_L)+\Phi(-c-d_U),
\qquad
P_N=\left[\Phi(d_U-c)-\Phi(c-d_L)\right]_+.
\label{eq:asymprob}
\end{equation}
The probability of negligibility is positive if and only if $d_L+d_U>2c$. Because $d_L+d_U=2\lambda$, this becomes $\lambda>c$. Thus fixed asymmetry under a symmetric confidence set introduces no fundamentally new finite-sample geometry: recentering at $m_n$ restores the symmetric formulas. The substantive complication appears when the two boundaries move at different rates.

\begin{proposition}[Total width does not determine resolvability under differential boundary rates]
\label{prop:asym}
Suppose $s_n=n^{-1/2}$,
\[
\Delta_{L,n}\asymp n^{-\gamma_L},
\qquad
\Delta_{U,n}\asymp n^{-\gamma_U},
\]
and $\gamma_L\ne\gamma_U$. Then $\lambda_n$ alone does not determine the limiting classification probabilities. The limits depend separately on $d_{L,n}$ and $d_{U,n}$.

For example, let $\theta_n=0$, $\gamma_L=0.9$, and $\gamma_U=0.2$. Then $\lambda_n\to\infty$, but
\[
d_{L,n}=n^{1/2-0.9}\to0,
\qquad
d_{U,n}=n^{1/2-0.2}\to\infty.
\]
Consequently,
\[
P_M\to\Phi(-c),
\qquad
P_N\to1-\Phi(c),
\qquad
P_U\to1-2\Phi(-c).
\]
For $c=z_{1-\alpha/2}$, these limits are $\alpha/2$, $\alpha/2$, and $1-\alpha$, respectively.
\end{proposition}

Thus $\lambda_n\to\infty$ does not guarantee negligibility classification when one scientific boundary remains unresolved. The appropriate coordinates for moving asymmetric regions are $(d_{L,n},d_{U,n})$, or equivalently $(u_n,\lambda_n)$, rather than $\lambda_n$ alone.

If both relevance boundaries converge to zero, then $m_n\to0$, and quadratic-mean differentiability at zero is appropriate. If the relevance region instead converges to a fixed nonzero center $m$, the corresponding quadratic-mean differentiability and Fisher-information assumptions should be imposed at $m$.

\section{Beyond the square-root rate: what does and does not generalize}

The achievability result is more general than regular Gaussian asymptotics. Suppose an estimator has uniform resolution
\[
|\hat\theta_n-\theta|=O_{P_\theta}(a_n^{-1})
\]
over the relevant parameter classes, where $a_n\to\infty$. Then Proposition~\ref{prop:uniform} immediately yields the sufficient condition
\begin{equation}
a_n\Delta_n\to\infty
\label{eq:generalrate}
\end{equation}
for consistent separated relevance classification.

For regular estimators, $a_n=\sqrt n$, which returns $\Delta_n\gg n^{-1/2}$. For a cube-root estimator, $a_n=n^{1/3}$, so achievability instead requires $\Delta_n\gg n^{-1/3}$. An inferential target that would be easily resolved by a regular estimator may therefore remain unresolved under a slower nonregular rate.

However, the converse requires more care. An estimator's convergence rate alone cannot establish an information-theoretic impossibility result. The lower bound in Section~\ref{sec:boundary} arises from distances between the underlying statistical experiments. In nonregular settings, a matching lower boundary requires a Hellinger, contiguity, or limit-experiment analysis for the particular statistical model.

The general lesson is asymmetric:
\[
\text{estimation rate gives achievability; experiment distance gives impossibility.}
\]
This distinction prevents estimator-specific convergence results from being incorrectly interpreted as minimax information bounds.

\section{Significance saturation across heterogeneous true effects}

The previous sections concern a single scientific parameter. A second phenomenon appears across a population of studies or effects. Let $\Theta\sim G$ represent a distribution of true effects.

Suppose a point-null testing procedure rejects $H_0:\Theta=0$ with conditional probability $q_n(\theta)$, and assume
\[
q_n(\theta)\to1\quad\text{for every fixed }\theta\ne0,
\qquad
q_n(0)\to\alpha.
\]
Let $S_n$ denote statistical significance, and define the \emph{irrelevant discovery share}
\begin{equation}
IDS_n=P(|\Theta|<\Delta\mid S_n).
\label{eq:ids}
\end{equation}
Let $p_0=P(\Theta=0)$.

\begin{theorem}[Significance-saturation limit]
\label{thm:ids}
Under the conditions above,
\begin{equation}
IDS_n\longrightarrow
\frac{P(0<|\Theta|<\Delta)+\alpha p_0}
{1-p_0+\alpha p_0}.
\label{eq:idslimit}
\end{equation}
If $G$ is continuous, then $p_0=0$ and
\[
IDS_n\longrightarrow P(|\Theta|<\Delta).
\]
\end{theorem}

\begin{proof}
By conditioning on $\Theta$,
\[
P(S_n)=\int q_n(\theta)\,dG(\theta).
\]
Because $0\le q_n(\theta)\le1$, dominated convergence gives
\[
P(S_n)\to1-p_0+\alpha p_0.
\]
Similarly,
\[
P(|\Theta|<\Delta,S_n)
=
\int_{|\theta|<\Delta}q_n(\theta)\,dG(\theta)
\]
converges to $P(0<|\Theta|<\Delta)+\alpha p_0$. Taking the ratio yields \eqref{eq:idslimit}. The continuous case follows when $p_0=0$.
\end{proof}

\subsection{Interpretation}

In a sufficiently precise environment, nearly every fixed nonzero effect is statistically detectable. Statistical significance therefore ceases to select effects mainly by their distance from zero. Instead, the distribution of significant effects approaches the distribution of nonzero effects already present in the scientific environment.

If many true effects lie inside $(-\Delta,\Delta)$, then many statistically significant results will eventually lie there as well. This is not a statement about false positives: these effects are genuinely nonzero. It is a statement about the difference between truth as departure from zero and importance as departure from a relevance region.

The finite-information path to this limit is not generally monotone. For the Gaussian point-null rule used in Figure~\ref{fig:ids}, consider first a low-information limit in which the statistical scale $s\to\infty$. For every fixed $\theta$, the rejection probability satisfies $q_s(\theta)\to\alpha$, so significance becomes nearly independent of the true effect. Dominated convergence then gives
\[
P(|\Theta|<\Delta\mid S_s)\longrightarrow P(|\Theta|<\Delta).
\]
For a continuous effect distribution, Theorem~\ref{thm:ids} gives the same value in the opposite high-information limit $s\to0$. Thus point-null significance can be least informative about scientific relevance at both extremes: at very low information it selects largely through sampling noise, whereas at very high information it selects nearly every nonzero effect. At intermediate information, significance is more selective for large $|\theta|$.

This mechanism produces the pronounced U-shaped curves in Figure~\ref{fig:ids}. With $\Delta=0.20$, the $N(0,0.10^2)$ curve reaches a minimum of approximately $0.822$ near $n=98$, compared with its limiting value $0.955$; the $N(0,0.30^2)$ curve reaches approximately $0.151$ near $n=37$, compared with $0.495$; and the Laplace distribution with scale $0.15$ reaches approximately $0.323$ near $n=48$, compared with $0.736$. These numerical minima are distribution-specific rather than a universal U-shape theorem, but they illustrate a general selection mechanism: there can be an intermediate information level at which point-null significance is most selective against scientifically negligible effects.

\begin{figure}[H]
\centering
\includegraphics[width=0.96\linewidth]{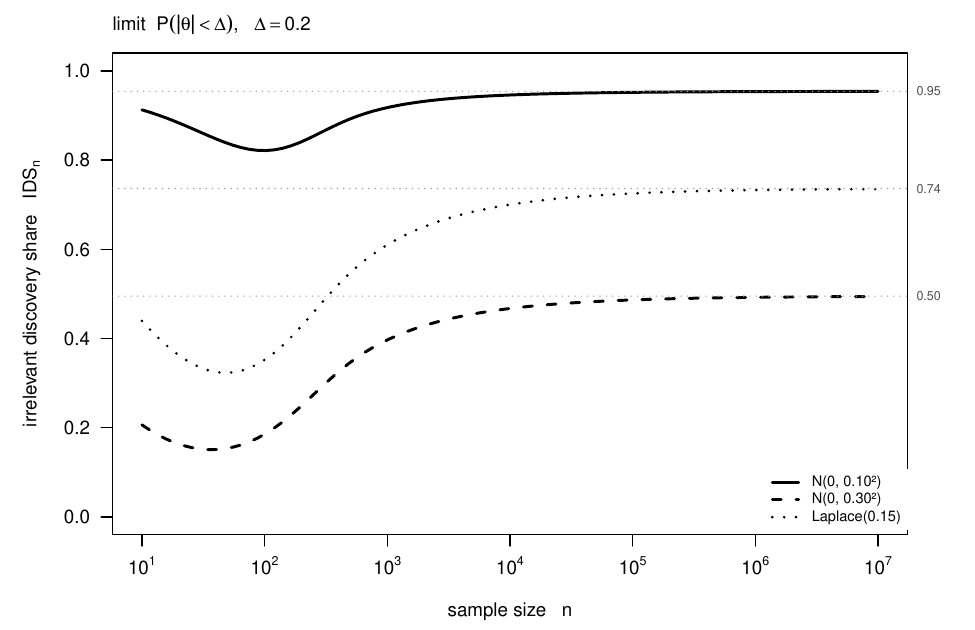}
\caption{\textbf{Intermediate-information selectivity and significance saturation.} The irrelevant discovery share $IDS_n=P(|\Theta|<\Delta\mid S_n)$ for several continuous effect distributions. The low-information limit, in which rejection is nearly independent of $\Theta$, and the high-information limit in Theorem~\ref{thm:ids} both equal $P(|\Theta|<\Delta)$ for continuous $G$. The numerical curves are U-shaped: at intermediate information, point-null significance is most selective for effects outside the scientific relevance region. The displayed minima are distribution-specific and are not asserted to be universal.}
\label{fig:ids}
\end{figure}

\section{Discussion}

Statistical power, effect size, and scientific relevance are often discussed as separate concepts. The present results show that they can instead be organized around a common question of resolution.

A regular statistical experiment has a natural local scale of order $n^{-1/2}$. Scientific interpretation supplies another scale, $\Delta_n$. Their Fisher-standardized ratio, $\lambda_n=\sqrt{nI_0}\Delta_n$, determines whether the scientific distinction can be learned from the experiment.

The minimax inequality in Theorem~\ref{thm:minimax} describes the loss of resolvability continuously over finite local information. When $\lambda_n\to0$, negligible and meaningful parameter classes merge statistically. No improved test, alternative confidence interval, Bayesian summary, or more elaborate classifier can recover information the experiment itself does not contain.

When $\lambda_n\to\kappa\in(0,\infty)$, the problem contains information, but only a finite amount on the local scale. Theorem~\ref{thm:minimax} prevents error from vanishing, while Proposition~\ref{prop:critical-exact} characterizes exactly how the separated composite problem is resolved in the Gaussian limit experiment. The optimal cutoff has a revealing endpoint behavior: $c^*\to1$ as $\kappa\downarrow0$, so the minimax rule retains a cutoff of one statistical-error unit even while the scientific boundary collapses toward zero. This yields a quadratic improvement over trivial risk, whereas the simpler scientific-boundary rule improves only cubically near that unresolved edge. At large $\kappa$, however, $c^*/\kappa\to1$ and the simple rule is asymptotically minimax. The exact finite-$\kappa$ calculation therefore complements rather than replaces the phase-transition result.

When $\lambda_n\to\infty$, scientific relevance becomes increasingly coarse relative to statistical uncertainty, and a uniformly rate-consistent estimator can classify effects lying a fixed fractional distance from the scientific boundary with probability tending to one. This is the sense in which the $n^{-1/2}$ relevance scale is a statistical resolution boundary.

\subsection{Relation to equivalence and relevant-hypothesis testing}

The framework is complementary to equivalence testing. Equivalence procedures answer questions such as whether the data support $|\theta|<\Delta$. Minimum-effect procedures ask whether the data support $|\theta|>\Delta$. Relevant-hypothesis procedures similarly construct tests around scientifically specified nonzero margins.

The present framework instead asks whether the chosen margin is statistically resolvable at the available information scale. The distinction is analogous to separating specification of a hypothesis from the information geometry of the experiment used to test it.

Romano's optimality theory is especially relevant here. Once the problem converges to a finite Gaussian experiment, optimal testing questions can be asked inside that experiment. The minimax-separation formulation of \citet{blanchard2018} is also close in decision criterion: it controls the worst-case sum of the two error probabilities while varying the geometric separation between composite hypotheses. The present framework adds the scientific-threshold interpretation and tracks that separation relative to the experiment's information scale. Accordingly, the exact Gaussian calculation in Proposition~\ref{prop:critical-exact} is best read as the critical-regime decision problem nested inside the broader relative-rate framework, not as a claim that Gaussian threshold optimization itself is new.

\subsection{Why large samples are not the problem}

Nothing in the theory implies that high precision is undesirable. If $0<|\theta|<\Delta$, then sufficiently precise data permit both $\theta\ne0$ and $|\theta|<\Delta$ to be established. The first statement is statistical. The second is scientific relative to the chosen relevance threshold. The difficulty arises only when rejection of a point null is interpreted as evidence for a different claim about substantive importance. More information improves our ability to distinguish those claims.

The heterogeneous-effect analysis adds a second qualification. Point-null significance need not become progressively less informative about scientific relevance as sample size increases. At very low information, significance is approximately independent of the true effect and therefore barely filters the population of effects. At very high information, it again filters little because nearly every nonzero effect is detected. The strongest enrichment for effects outside the relevance region can occur between these extremes. The U-shaped examples in Figure~\ref{fig:ids} therefore represent a selection phenomenon, not merely slow convergence to an asymptote.

\subsection{The relevance threshold is part of the scientific problem}

The framework does not determine $\Delta_n$. The relevance threshold must come from domain knowledge, decision consequences, measurement considerations, theory, policy objectives, or another scientifically justified criterion.

For some problems, a fixed threshold is natural. For others, $\Delta_n$ may move with measurement technology, population scale, cost, scientific ambition, or the resolution of substantive theory. Allowing $\Delta_n$ to vary is therefore not merely a mathematical device. It formalizes the possibility that the scientific target itself becomes increasingly demanding. The theory then asks whether statistical information improves quickly enough to keep pace.

\subsection{Boundary distance is more fundamental than region width}

The asymmetric analysis shows why a single width parameter can sometimes mislead. For $[-\Delta_{L,n},\Delta_{U,n}]$, the scientific decision at $\theta_n$ depends on $d_{L,n}$ and $d_{U,n}$. A large total width does not imply that both scientific boundaries are resolved. This distinction is likely to matter in applications involving noninferiority, asymmetric clinical margins, one-sided policy costs, or scientific regions whose two boundaries evolve differently with design or measurement.

\subsection{Limitations and extensions}

Several limitations are important. First, the minimax lower bound is derived for an i.i.d. scalar model that is quadratic-mean differentiable at the local center. Extensions to dependent observations, semiparametric models, and nonregular experiments require model-specific experiment-distance arguments.

Second, Proposition~\ref{prop:critical-exact} solves the Gaussian \emph{limit} experiment. Quadratic-mean differentiability and pointwise LAN are sufficient for the lower-bound transfer used in Theorem~\ref{thm:minimax}, but they do not by themselves imply $R_n^*\to R_\infty^*$ for the full composite classes. Such an upper transfer requires uniform local approximation together with control ensuring that alternatives outside bounded local neighborhoods are no harder. Corollary~\ref{cor:gaussian-finite} shows exact finite-sample attainability in the Gaussian location model, but we do not assert it for every regular model.

Third, the Uniform Resolution Condition remains a substantive assumption. The Gaussian and Bernoulli calculations above verify it in two basic models, but pointwise asymptotic normality alone does not imply the uniform control needed for Proposition~\ref{prop:uniform}; irregular and superefficient behavior can invalidate such a step.

Fourth, the scientific relevance region is treated as prespecified. If $\Delta_n$ is selected using the same data employed for inference, additional selection effects arise.

Fifth, a hard threshold is not always an adequate representation of scientific value. Some applications are better represented through continuous utility or loss functions. In those settings, the present relevance region can be interpreted as a simple decision boundary rather than a complete substantive utility model.

Finally, the theory here is scalar. Higher-dimensional relevance classification introduces genuinely different geometric and minimax questions. In particular, detection from a point null and classification relative to a composite relevance region need not share the same high-dimensional boundary, so such extensions require a separate minimax analysis.

\section{Conclusion}

The statistical significance of an effect and the scientific importance of an effect are governed by different distances. Point-null evidence is driven by $\theta/s_n$. Scientific relevance is driven by distance from the boundaries of a relevance region.

For a symmetric scientific threshold $\Delta_n$, the central scale is $\lambda_n=\Delta_n/s_n$. In a regular experiment this becomes $\lambda_n=\sqrt{nI_0}\Delta_n$.

The resulting classification problem obeys a statistical-resolution law. If $\lambda_n\to0$, meaningful and negligible effects can become statistically indistinguishable. If $\lambda_n\to\kappa\in(0,\infty)$, the problem retains finite but nonvanishing uncertainty: meaningful discrimination is possible, but consistency is not. In the Gaussian limit experiment the corresponding composite minimax rule is an absolute-value threshold $|Y|>c^*$, with $c^*\to1$ at the unresolved edge and $c^*/\kappa\to1$ as resolution improves. If $\lambda_n\to\infty$, scientific relevance can be classified consistently under suitable uniform estimation conditions.

For regular models, the boundary occurs at $\Delta_n\asymp n^{-1/2}$. Thus the central-limit scale is not only the scale of estimation error. It is also the scale separating scientific distinctions that the experiment can resolve from distinctions that it cannot.

The resulting principle is simple:
\[
\text{statistical evidence is distance measured in uncertainty units;}
\]
\[
\text{scientific relevance is distance measured from a substantive boundary.}
\]
A scientifically meaningful inferential analysis must keep both scales visible.

\clearpage
\begin{center}
{\Large\bfseries References}
\end{center}
\vspace{0.6em}
\printbibliography[heading=none]

\clearpage
\begin{center}
{\Large\bfseries Appendix}
\end{center}
\vspace{0.6em}

\subsection*{A.1 Total variation and binary testing}

Throughout the paper,
\[
\|P-Q\|_{\mathrm{TV}}=\sup_A|P(A)-Q(A)|.
\]
For every measurable $\phi:\mathcal X\to[0,1]$,
\[
|E_P\phi-E_Q\phi|\le\|P-Q\|_{\mathrm{TV}}.
\]
For testing two simple hypotheses $P$ versus $Q$,
\[
\inf_\phi\{E_P\phi+E_Q(1-\phi)\}
=
1-\|P-Q\|_{\mathrm{TV}}.
\]
This identity underlies the two-point minimax reduction in Theorem~\ref{thm:minimax}.

\subsection*{A.2 Gaussian total variation calculation}

Let $P=N(\mu_1,1)$ and $Q=N(\mu_2,1)$ with $\mu_2>\mu_1$. The densities cross at $(\mu_1+\mu_2)/2$. Consequently,
\[
\|P-Q\|_{\mathrm{TV}}
=
2\Phi\left(\frac{\mu_2-\mu_1}{2}\right)-1.
\]
For $\mu_1=(1-\varepsilon)\kappa$ and $\mu_2=(1+\varepsilon)\kappa$, this becomes $2\Phi(\varepsilon\kappa)-1$.

\subsection*{A.3 Hellinger proof at the unresolved boundary}

With the normalization $h^2(P,Q)=1-\int\sqrt{dP\,dQ}$, quadratic-mean differentiability at zero gives
\[
h^2(P_\theta,P_0)=\frac18I_0\theta^2+o(\theta^2).
\]
For independent products,
\[
h^2(P_\theta^n,P_0^n)=1-\{1-h^2(P_\theta,P_0)\}^n.
\]
Since $1-(1-x)^n\le nx$,
\[
h^2(P_{\theta_n}^n,P_0^n)
\le
n h^2(P_{\theta_n},P_0).
\]
If $\sqrt n\theta_n\to0$, the right-hand side converges to zero. Hence $h(P_{\theta_n}^n,P_0^n)\to0$, and $\|P-Q\|_{\mathrm{TV}}\le\sqrt2\,h(P,Q)$ implies merging in total variation.

\subsection*{A.4 Expansions for the exact critical-regime minimax rule}

Write $a=(1-\varepsilon)\kappa$ and $b=(1+\varepsilon)\kappa$. Taking logarithms in \eqref{eq:cstar} gives
\begin{equation}
\log\cosh(bc^*)-\log\cosh(ac^*)
=\frac{b^2-a^2}{2}
=2\varepsilon\kappa^2.
\label{eq:appendix-cstar-log}
\end{equation}
For $\kappa\downarrow0$, use
\[
\log\cosh z=\frac{z^2}{2}-\frac{z^4}{12}+O(z^6)
\]
and write $c^*=1+d\kappa^2+O(\kappa^4)$. Since
\[
b^2-a^2=4\varepsilon\kappa^2,
\qquad
b^4-a^4=8\varepsilon(1+\varepsilon^2)\kappa^4,
\]
substitution into \eqref{eq:appendix-cstar-log} yields
\[
2\varepsilon\kappa^2
+\left\{4\varepsilon d-\frac{2}{3}\varepsilon(1+\varepsilon^2)\right\}\kappa^4
+O(\kappa^6)
=2\varepsilon\kappa^2.
\]
Therefore $d=(1+\varepsilon^2)/6$, proving \eqref{eq:cstar-small}.

For the risk expansion, the symmetric mixture density satisfies, for fixed $y$,
\[
\frac{g_m(y)}{\varphi(y)}
=e^{-m^2/2}\cosh(my)
=1+\frac{m^2}{2}(y^2-1)+O(m^4).
\]
Because the minimax risk equals the equal-prior Bayes sum-error risk,
\[
R_\infty^*=1-\|g_b-g_a\|_{\mathrm{TV}}.
\]
Moreover,
\[
g_b(y)-g_a(y)
=2\varepsilon\kappa^2\varphi(y)(y^2-1)+O(\kappa^4),
\]
with an integrable remainder after multiplication by a Gaussian envelope. Hence
\[
\|g_b-g_a\|_{\mathrm{TV}}
=\varepsilon\kappa^2
\int_{-\infty}^{\infty}|y^2-1|\varphi(y)\,dy
+O(\kappa^4).
\]
Since
\[
\int_{-\infty}^{\infty}|y^2-1|\varphi(y)\,dy=4\varphi(1),
\]
we obtain
\[
R_\infty^*(\kappa,\varepsilon)
=1-4\varepsilon\varphi(1)\kappa^2+O(\kappa^4),
\]
which is \eqref{eq:rstar-small}.

For the large-$\kappa$ cutoff, use
\[
\log\cosh z=z-\log2+\log(1+e^{-2z}),\qquad z>0.
\]
Equation \eqref{eq:appendix-cstar-log} becomes
\[
2\varepsilon\kappa(c^*-\kappa)
=
\log(1+e^{-2ac^*})-\log(1+e^{-2bc^*}).
\]
The right-hand side is positive, so $c^*>\kappa$, and
\[
0<c^*-\kappa
\le
\frac{e^{-2ac^*}}{2\varepsilon\kappa}
\le
\frac{e^{-2(1-\varepsilon)\kappa^2}}{2\varepsilon\kappa}.
\]
Thus $c^*/\kappa\to1$.

\subsection*{A.5 The scientific-boundary rule and large-$\kappa$ comparison}

For $\phi_\kappa(Y)=\mathbf1\{|Y|>\kappa\}$, the false-positive probability at the least favorable negligible boundary $\mu=(1-\varepsilon)\kappa$ is
\[
\Phi(-\varepsilon\kappa)+\Phi\{-(2-\varepsilon)\kappa\}.
\]
At the least favorable meaningful boundary $\mu=(1+\varepsilon)\kappa$, the false-negative probability is
\[
\Phi(-\varepsilon\kappa)-\Phi\{-(2+\varepsilon)\kappa\}.
\]
Their sum is equation \eqref{eq:S}. Furthermore,
\[
\begin{aligned}
1-S_\varepsilon(\kappa)
&=
\int_{-\varepsilon\kappa}^{\varepsilon\kappa}\varphi(x)\,dx
-
\int_{(2-\varepsilon)\kappa}^{(2+\varepsilon)\kappa}\varphi(t)\,dt\\
&=
\int_{-\varepsilon\kappa}^{\varepsilon\kappa}
\{\varphi(x)-\varphi(x+2\kappa)\}\,dx>0.
\end{aligned}
\]
Taylor expansion at zero gives equation \eqref{eq:cubic}.

Finally, the two-point argument used in Theorem~\ref{thm:minimax} gives, directly in the Gaussian experiment,
\[
2\Phi(-\varepsilon\kappa)
\le R_\infty^*(\kappa,\varepsilon)
\le S_\varepsilon(\kappa).
\]
But
\[
S_\varepsilon(\kappa)-2\Phi(-\varepsilon\kappa)
=
\Phi\{-(2-\varepsilon)\kappa\}
-
\Phi\{-(2+\varepsilon)\kappa\},
\]
and this difference is $o\{\Phi(-\varepsilon\kappa)\}$ because $2-\varepsilon>\varepsilon$ for $0<\varepsilon<1$. The squeeze theorem therefore gives
\[
R_\infty^*(\kappa,\varepsilon)
\sim S_\varepsilon(\kappa)
\sim2\Phi(-\varepsilon\kappa),
\]
proving Corollary~\ref{cor:largekappa}.

\subsection*{A.6 Asymmetric classification formulas}

For the asymmetric region, $d_L=\lambda+u$ and $d_U=\lambda-u$. Meaningfulness occurs when either the lower end of the confidence interval exceeds the upper relevance boundary or the upper end lies below the lower relevance boundary. Under standardized Gaussian error $Z$,
\[
P_M=P(Z>d_U+c)+P(Z<-d_L-c),
\]
giving the first expression in \eqref{eq:asymprob}. Negligibility requires
\[
c-d_L<Z<d_U-c,
\]
which yields the second expression in \eqref{eq:asymprob}. For the sequence in Proposition~\ref{prop:asym}, $d_{L,n}\to0$ and $d_{U,n}\to\infty$, giving the displayed limits.

\end{document}